\documentclass[12pt,a4paper]{amsart}
\usepackage{graphicx,amssymb,latexsym}
\input xy
\xyoption{all}
\usepackage[all]{xy}

\usepackage{color}
\usepackage{pb-diagram}
\usepackage{epstopdf}

\newtheorem{thm}{Theorem}[section]

\newtheorem{cor}[thm]{Corollary}
\newtheorem{lem}[thm]{Lemma}
\newtheorem{prop}[thm]{Proposition}
\theoremstyle{definition}
\newtheorem{defn}[thm]{Definition}
\newtheorem{rem}[thm]{Remark}

\numberwithin{equation}{section}
\newcommand{\RR}{\mathbb R}
\newcommand{\QQ}{\mathbb Q}
\newcommand{\NN}{\mathbb N}
\newcommand{\ZZ}{\mathbb Z}
\newcommand{\CC}{\mathbb C}
\newcommand{\PP}{\mathbb P}
\newcommand{\FF}{\mathbb F}

\newcommand{\lra}{\longrightarrow}

\newcommand{\ra}{\rightarrow}

\newcommand{\tC}{\widetilde{C}}

\newcommand{\cO}{\mathcal{O}}

\newcommand{\ch}{\mathfrak{h}}

\DeclareMathOperator{\Aut}{{Aut}}
\DeclareMathOperator{\Hom}{{Hom}}
 
 \DeclareMathOperator{\Ker}{Ker}
  \DeclareMathOperator{\Fix}{Fix}

 \DeclareMathOperator{\Nm}{{Nm}}

 \DeclareMathOperator{\diag}{{diag}}

 \DeclareMathOperator{\im}{Im}
 
\DeclareMathOperator{\End}{{End}}

\DeclareMathOperator{\mult}{mult}
\DeclareMathOperator{\id}{id}

\newcommand{\Th}[2]{{\theta\genfrac{[}{]}{0pt}{-1}{#1}{#2}}}

\begin{document}

\title[ ]{ Hyperelliptic curves on $(1,4)$ polarised abelian surfaces}
\author{ Pawe\l{} Bor\'owka,  Angela Ortega}
\address{P. Bor\'owka \\ Institute of Mathematics, Jagiellonian University in Krak\'ow, Poland}
\email{Pawel.Borowka@uj.edu.pl}
              
\address{A. Ortega \\ Institut f\" ur Mathematik, Humboldt Universit\"at zu Berlin \\ Germany}
\email{ortega@math.hu-berlin.de}

%\thanks{}
%\keywords{}%

\date{\today }
%\dedicatory{ }%
%\commby{ }%
% ----------------------------------------------------------------
\begin{abstract} 
We investigate the number and the geometry of smooth hyperelliptic
curves on a general complex abelian surface. We show that the only
possibilities of genera of such curves are $2,3,4$ and $5$.  We focus on the genus 5 case.
We prove that up to translation, there is a unique hyperelliptic curve in the linear system of
a general $(1,4)$ polarised abelian surface. Moreover, the curve is
invariant with respect to a subgroup of translations isomorphic to the
Klein group. We give the decomposition of the
Jacobian of such a curve into abelian subvarieties displaying Jacobians of quotient curves and Prym varieties.
Motivated by the construction, we prove the statement: every \'etale Klein covering of a
hyperelliptic curve is a hyperelliptic curve, provided that the group of $2$-torsion points defining the covering is non-isotropic with respect to the Weil 
pairing  and every element of this group can be written as a difference of two Weierstrass points. 
\end{abstract}

\maketitle
%\tableofcontents
\section{Introduction}
In the research connecting curves and abelian varieties an special attention has been given to abelian surfaces, because then 
the curves become divisors. Many classical results focus on principally polarised abelian surfaces, which are mostly Jacobians of smooth genus 2 curves. 
In particular, the linear series of the polarising line bundle on a principally polarised surface contains a unique section (up to a scalar) with zero locus being a
smooth genus 2 curve and every genus 2 curve arises in this way.
% Moreover, we can assume that the curve is symmetric, hence descends to the Kummer
%surface and we recover a beautiful $(16_6)$-configuration. 

Contrary to these results, there is little known for general $(d_1,d_2)$ polarised abelian surfaces. The curves in the linear series form a 
$(d_1d_2-1)$-dimensional family with a general member being a smooth curve of genus $1+d_1d_2$. A well known example of a curve in the linear series
of the polarisation is an \'etale  cyclic $d:1$ covering of  a genus 2 curve, that is embedded in a $(1,d)$ polarised surface.

One can consider symmetric line bundles and symmetric curves with respect to the  $(-1)$-action of the abelian variety.  Using the  projection map to the
Kummer surface we see that such curves, by definition admit a $2:1$ map branched at the 2-torsion points, through which the curve passes with odd 
multiplicity. Conversely, every hyperelliptic curve that is embedded into an abelian surface can be mapped in such a way that the image is symmetric. 
This fact has been recently used in \cite{BOPY}, to compute the number of hyperelliptic curves on abelian surfaces using Gromov-Witten theory on the 
Kummer surfaces. In particular they proved that these numbers are finite; the details are in \cite[Table 1]{BOPY}.

Independently, in \cite{BS} there is a construction of the so called $(1,3)$-theta divisors on $(1,3)$ polarised abelian surfaces, that are hyperelliptic curves. 
Using the $(-1)$-action on a general $(1,3)$ polarised surface, one can decompose the linear series of the symmetric polarising line bundle into $\pm 1$ 
eigenspaces of dimensions $1$ and $2$, and therefore distinguish a unique theta function and its zero locus, that is a genus 4 hyperelliptic curve.

The motivation of this paper has been to complete the study of which and how many smooth hyperelliptic curves can be embedded into a general abelian
 surface. The starting point of the investigation is Theorem \ref{poss-genera}, where we prove that the necessary condition is that the genus 
 $g\in\{2,3,4,5\}$ and the surface is polarised of type $(1,g-1)$. Since the cases $g=2$ and $g=3$ are classical and $g=4$ is treated in \cite{BS}, we focus 
 on the case $g=5$ embedded into $(1,4)$ polarised surfaces. Using a similar construction to the one in \cite{BS}, yet technically more difficult, we prove the
  following theorem.
\vspace{0.2cm}

\textbf{ Theorem \ref{C_A}.} \textit{ 
Let $C$ be a smooth hyperelliptic curve of genus 5 embedded in a general $(1,4)$ polarised abelian surface $A$. Then $C$ is a  translation
of the curve $C_A$, defined as the zero locus of the unique (up to multiplication by a scalar) odd theta function $\theta_A$ 
on $A$.}\vspace{0.2cm}

In particular, we show that the conditions in Theorem \ref{poss-genera} are sufficient. Moreover, we explicitly constructed all the smooth 
curves predicted  in \cite[Table 1]{BOPY}.

Apart from studying the number of such hyperelliptic curves, we investigated their geometry. We prove that the curve $C_A$ is invariant with respect 
to the Klein subgroup of the group of translations of $A$. This led to the investigation of Klein coverings of hyperelliptic curves. Using the theory of coverings
we prove the following theorem.
\vspace{0.2cm}

\textbf{Theorem \ref{thm:klehyp}.}
\textit{Let $H$ be a smooth hyperelliptic curve of genus $g\geq 2$ with a subgroup $G=\{0,\eta_1,\eta_2,\eta_1+\eta_2\} \subset JH[2]$ and let $\tC \ra H$
 be the Klein covering of (H,G).
 \begin{itemize}
 \item[(a)] If $\tC$ is hyperelliptic then $G$ is non-isotropic with respect to the Weil pairing on $JH[2]$.
 \item[(b)] If  $G$ %= \langle \eta_1, \eta_2 \rangle$ 
 is non-isotropic with $\eta_1$ and $\eta_2$ (and hence $\eta_1+\eta_2$) being the difference of two Weierstrass points,
 then $\tC$ is hyperelliptic.
 \end{itemize}}

This result is independent of the rest of the paper. We apply it to the following construction. Let $\widehat{A}=JH/G$ be the quotient surface. Lemma 
\ref{lem:polarisations} shows that $\widehat{A}$ is $(1,4)$ polarised so is its dual, called $A$. We have a map $\pi:A\ra JH$ and define $\tC=\pi^{-1}(H)$. 
By construction, $\pi|_{\tC}:\tC\ra H$ is an \'etale covering defined by a non-isotropic Klein subgroup $G$, hence hyperelliptic and embedded in $A$. As a
consequence of the above statements, we obtain the main result of the paper: 
\vspace{0.2cm}

\textbf{ Theorem \ref{thm:hyp}.}
A smooth hyperelliptic curve of genus 5 can be embedded into $(1,4)$ polarised abelian surface if and only if it is a non-isotropic \'etale Klein covering of a genus 2 curve.
\vspace{0.2cm}

As an application of the results we decompose the Jacobian of the curve $\tC$ into subvarieties. To make the statement precise, we need to introduce
 the following notation. If $M$ and $N$ are abelian subvarietes of $X$ with restricted polarisation of type $D$, respectively  $D'$, and with $\epsilon_M,
 \epsilon_N$ the associated idempotents, we write $X=M^D\boxplus N^{D'}$ if $\epsilon_M+\epsilon_N=1$. Moreover, let $\iota$ be the hyperelliptic involution 
 on $\tC$ and $\sigma, \tau$ the involutions of the Klein coverings. Define the elliptic curves  $E_{\sigma}=\tC/\langle \sigma, \iota\tau \rangle, \ 
 E_{\tau}=\tC/\langle \tau, \iota\sigma \rangle, \ E_{\sigma\tau}=\tC/\langle \sigma\tau, \iota\sigma \rangle$. 
\vspace{0.2cm}

\textbf{ Theorem \ref{thm:jacc_a}.}
Let $A$ be a general $(1,4)$ polarised abelian surface and $\tC$ the hyperelliptic curve embedded in $A$. Then
$$J\tC=\widehat{A}^{(1,4)}\boxplus E^{(4)}_{\sigma} \boxplus E^{(4)}_{\tau}\boxplus E^{(4)}_{\sigma\tau}.$$
In the above presentation, one can recognise images of Jacobians of all the quotient curves and Prym varieties of the quotient maps.
\vspace{0.2cm}

The plan of the paper is as follows. We recall some basic facts and definitions in Section~2. In Section 3 we construct 
smooth hyperelliptic curves on a general $(1,4)$ polarised abelian surface and prove Theorem \ref{C_A} via degeneration to a  product of two 
elliptic curves. In Section 4 we prove Theorem \ref{thm:klehyp} by means of a result in \cite{F} 
and give a second construction  of hyperelliptic curves on a $(1,4)$ polarised abelian surface that leads to Theorem \ref{thm:hyp}. 
In the last section we give the decomposition of the Jacobian of $\tC$ in terms of the abelian subvarieties 
that it contains.

\section{Preliminaries}

We  recall very briefly some definitions and known results from theories of abelian varieties and coverings of curves. For details, we refer to \cite{BL}.
Let $X$ be a complex abelian variety of dimension $g$, i.e. a projective complex torus. The Neron-Severi class of an ample line bundle $L$ will be called 
a polarisation. The dual abelian variety to $X$ will be denoted by $\widehat{X}$. A polarisation induces an isogeny $\phi_L: X \ra \widehat{X}$. 
The kernel of $\phi_L$,  denoted by $K(L)$ is isomorphic to $(\ZZ^g/D\ZZ^g)^2$, where $D=\diag(d_1,\ldots,d_g)$ for some $d_i\in\ZZ_+, d_i|d_{i+1}$ 
and is called the type of the polarisation.

Let $i:Y \hookrightarrow X$ be the embedding of an abelian subvariety $Y$.  The polarisation $L$ induces an isogeny $\psi_{i^*L} : \widehat{Y} \ra Y$ and the 
{\it norm endomorphism} $\Nm_Y \in \End (X)$  is defined by 
$$
\Nm_Y = i \psi_{i^*L} \hat{i} \phi_L.
$$ 
 One can also associate to $Y$ the symmetric idempotent $\epsilon_Y := \frac{1}{e(Y)}\Nm_Y \in 
\End_{\QQ}(X) $, where $e(Y)$ is the exponent of $Y$. Conversely, if $\epsilon$ is a symmetric idempotent, then there exists an integer $n>0$ such that 
$n\epsilon \in \End(X)$ and $\im (n\epsilon)$ defines an abelian subvariety of $X$. This gives a one-to-one correspondence between abelian 
subvarieties of $X$ and symmetric idempotents (\cite[Theorem 5.3.2]{BL}). The map $\epsilon \mapsto 1-\epsilon$ induces an involution on the set 
of symmetric idempotents, therefore   we have an involution  on the set of abelian subvarieties of $X$. If $Y=\im (n\epsilon)$, then $Z=\im(n(1-\epsilon))$
is the {\it complementary abelian subvariety} to $Y$ with respect to the polarisation $L$.

For a smooth curve $C$, by $JC$, we denote its Jacobian that is canonically a principally polarised (i.e. $D=\id_g$) abelian variety of dimension equal 
the genus $g=g(C)$.
If $f:\tC\ra C$ is a covering then $\im f^*$ is an abelian subvariety of $J\tC$ and the complementary  subvariety is called the {\it Prym variety} of the 
covering and will be denoted by $P(\tC/C)$.
%We will denote by $K(L)$ the kernel of the polarisation $\phi_L$.  

Recall that there is a bijective correspondence between subgroups $\langle \eta \rangle$ of order $n$ of $JC$ and \'etale cycle coverings 
$f: \tC \ra C$  of degree  $n$ (see \cite{mu}). One can also realise  \'etale cyclic coverings in the following way.  
 \begin{prop} \label{prop:cyccov}
 Let $\widehat{X}=JC/\langle \eta \rangle$ and let  $\pi:JC\ra \widehat{X}$ be the quotient isogeny.
Then $\widehat{X}$ can be embedded in $J\tC$ and the restricted polarisation from $J\tC$ to $\hat{X}$ is of type $(1,n,\ldots,n)$. 
If  $\widehat{\pi}: X\ra JC$ is the dual isogeny, then  $\widehat{\pi}^{-1}(C)$ is a curve $\tC$ and $\widehat{\pi}|_{\tC}=f$ is the given covering. 
\end{prop}
\begin{proof}
(\cite[Section 3]{mu}).  
\end{proof}

Let $f:\tC\ra C$ be a double covering and $\sigma$ the involution exchanging the sheets of the covering; denote also by $\sigma$ the induced 
automorphism on $J\tC$.  The following Proposition is a well-known fact (\cite[Section 3]{mu}).
\begin{prop}\label{prop:doubra}
If $f:\tC\ra C$ is a branched double covering then $f^*:JC\ra J\tC$ is an embedding and the restricted polarisation is twice the principal one. 
Moreover,  the norm map of $JC$ is  $\Nm_{JC}=1+\sigma$. 
\end{prop}

If $\alpha$ is an automorphism of a curve $C$ we denote by $C_{\alpha}$ the quotient curve $C/\langle \alpha\rangle$.
Recall that for a hyperelliptic curve of genus $g\geq 2$, the set of Weierstrass points coincide with the set of fixed points of the hyperelliptic involution and is of order $2g+2$.
The following propositions deals with the case when the  curves are hyperelliptic.
\begin{prop}\label{prop:hypcov}
Let $f:\tC \ra C$ be an \'etale cyclic covering of degree $n$. If $\tC$ is hyperelliptic then $C$ is also hyperelliptic and $n=2$.
\end{prop}
\begin{proof}
Let $\sigma$ be the automorphism permuting the sheets of the covering and let $\tilde \iota$ be the hyperelliptic involution. Obviously $\tilde \iota$ 
commutes with $\sigma$ so it descends to an involution on $C$ denoted by $\iota $. Let $g=g(C)$ be the genus of $C$. By Hurwitz formula 
$g(\tC)=n(g-1)+1$, hence $|\Fix(\tilde \iota)|= 2n(g-1)+4$. Now, $|\Fix(\iota)| \geq \frac{|\Fix(\tilde \iota)|}{n}$ because 
$\tilde \iota$ descends to $\iota$. 
Since $|\Fix(\iota) |\geq 2(g-1)+\frac{4}{n}$, we can write $|\Fix(\iota)|= 2(g-1)+b$, for some positive integer $b$. 
Let us consider the quotient curve $C_\iota$.  By  Hurwitz formula we have:
$$
2g-2=2(2g(C_{\iota})-2)+|\Fix(\iota)|=2(2g(C_\iota)-2)+2g-2+b.
$$
Since $b >0$ the only possibility is that $g(C_\iota)=0$, hence $b=4$ and $C$ is hyperelliptic with the hyperelliptic involution $\iota$.
The last part follows from the fact that a lift of $\iota$ and $\sigma $ generate the dihedral group $D_n$  \cite[Proposition 1.2]{BL-Surfaces} 
and the only such a group that is abelian  is $D_2=\ZZ_2\times \ZZ_2$. 
\end{proof}

\begin{prop}\label{prop:douhyp}
Let $f:C\ra C_{\sigma}$ be a double covering and $\sigma$ be the involution exchanging the sheets of the covering. Then
\begin{itemize}
\item[(a)] If $f$ is \'etale and $C_{\sigma}$ is hyperelliptic with the hyperelliptic involution $\iota$, then $\iota$ lifts to an involution on 
$C$ denoted by $\tilde\iota$
and $P(C/C_{\sigma})$ splits as product of Jacobians $JC_{\tilde\iota}\times JC_{\tilde\iota\sigma}$.
\item[(b)] If $C$ is hyperelliptic with the hyperelliptic involution $\iota$, then $P(C/C_{\sigma})$ is the image of a Jacobian $JC_{\iota\sigma}$.
\end{itemize}
\end{prop}
\begin{proof}
Part $(a)$ can be found in \cite{mu}.
As for the second part, let $M$ be the image of $JC_{\sigma}$ in $JC$. By construction, the norm map of $M$ is given by $\Nm_M=1+\sigma$.

On the other hand if $\iota$ is the hyperelliptic involution on $C$ then $\iota\sigma$ is another involution on $C$.
If $N$ is the image of $JC_{\iota\sigma}$ in $JC$ then $\Nm_N=1+\iota\sigma$ and since $\iota$ extends to $JC$ as $(-1)$, 
we get that $\Nm_N=1 -\sigma$. Hence $\Nm_N+\Nm_M=2_{JC}$ so $(M,N)$ are complementary abelian subvarieties. 

Hence, by definition of Prym, we get that $\im(JC_{\iota\sigma})=N=P(C/C_{\sigma})$.
\end{proof}

\begin{rem}
If both $C$ and $C_{\alpha}$ are hyperelliptic and we choose the lift of the hyperelliptic involution to be hyperelliptic then part $(a)$ and $(b)$ of 
Proposition \ref{prop:douhyp} coincide, since $J\PP^1=0$.
\end{rem}

\subsection{Motivating computation}
Let $A$ be an abelian surface and $C$ be a smooth hyperelliptic curve of genus $g>1$.
Let $f_C\colon C\hookrightarrow A$ be an embedding. Without lose of generality, we can assume that $f_C(p)=0$ for some Weierstrass point $p\in C$. Using the Universal 
Property of Jacobians together with the Abel map $\alpha_p$, we have the following diagram:
\begin{equation*}
  \begin{diagram}
\label{Diag1}
\dgARROWLENGTH=2.5em
%	\node[3]{C''}\arrow[2]{se,t}{i_{C''}} %\node[4]{J(C')}
%	\\[2]
    \node{C}\arrow[4]{e,t}{f_C}\arrow[2]{se,r}{\alpha_p}%\arrow[2]{ne,l}{m|_C}
    %\node[2]{A}\arrow[2]{e,t}{m}
    \node[4]{A}\\[2]\node[3]{JC}%\arrow[2]{n,r}{f}
\arrow [2]{ne,r}{f}
    %\\[2]
%    \node[1]{K^0}\arrow[2]{ne,r}{k}  
  \end{diagram}
\end{equation*}
where $f$ is the extension of $f_C$ to $JC$. Let $\iota$ be the hyperelliptic involution on $C$. Its extension to $JC$ is $-1$, so $\alpha_p(C)$ 
is a symmetric curve, i.e. such that $\alpha_p(C)=-\alpha_p(C)$. 

\begin{lem}\label{lem:WcapA2}
Let $W$ be the set of Weierstrass points on $C$. Then $f_C(W)=f_C(C)\cap A[2]$.
\end{lem}
\begin{proof}
First note that $f(JC[2]) = A[2]$. Since $f|_{\alpha_p(C)}$ is injective the statement is equivalent to the equality  $\alpha_p(W) = \alpha_p(C) \cap JC[2]$.
For $q \in C$,
$$
\alpha_p(q) \in JC[2] \ \Leftrightarrow \ -\alpha_p(q) =\alpha_p(\iota q) =  \alpha_p(q) \ \Leftrightarrow \ \iota q=q \ \Leftrightarrow \ q \in W. 
$$

\end{proof}

From now on, we identify $C$ with its image in $JC$ as well as with its image in $A$.
Since $f$ is a homomorphism,  $-f(C)=f(-C)=f(C)$ and the curve $C$ on $A$ defines a symmetric divisor. 
Since $g(C)=g>1$, we have that $\cO_A(f_C(C))$ is a polarising line bundle of some type $(d_1,d_2)$ with $g=1+d_1d_2$. 

For any divisor $D$ on $A$, let
\begin{align*}
A^+_2(D)&=\{a\in A[2]\colon \mult_a(D)\equiv 0 \mod 2\},\\
A^-_2(D)&=\{a\in A[2]\colon \mult_a(D)\equiv 1 \mod 2\}.
\end{align*}

The following proposition is a  simplified version of \cite[Proposition 4.7.5]{BL}.
\begin{prop}\label{prop:BL475}
Let $D$ be an ample symmetric divisor on an abelian surface $A$ and $L=\cO_A(D)$. Suppose $L$ is of type $(d_1,d_2)$ and 
$s=|\{i\colon d_i\in 2\NN+1\}|$ is a number of odd $d_i$'s. 
Then $$
|A^-_2(D)|=8+2^{3-s}\ \text{ or }\ |A^-_2(D)|=8-2^{3-s}\ \text{ or }\ |A^-_2(D)|=8.
$$\qed
\end{prop}
Since $C$ is smooth and $f_C$ is an embedding we have $\mult_af(C)\in\{0,1\}$ for any $a\in A$. 
The following theorem was the starting point of our research.
\begin{thm} \label{poss-genera}
Let $C$ be a smooth hyperelliptic curve embedded in an abelian surface $A$. Then its genus $g(C)\in\{2,3,4,5\}$ and $A$ is polarised of type $(1,g(C)-1)$.
\end{thm}
\begin{proof}
By Lemma \ref{lem:WcapA2}, the set of Weierstrass points is contained in the set of 2-torsion points and is equal to $A^-_2(C)$. Hence $2g(C)+2
\leq 16 =|A[2]|$, so $g(C)\leq7$. It remains to exclude two cases, $g(C)=6$ and $g(C)=7$. In the former case, we have that $|A^-_2(C)|=14$ which
is a contradiction with Proposition \ref{prop:BL475}. For $g(C)=7$ the only possible option is $d_1=1, d_2=6$, so $s=1$ and we also have a contradiction.

For $g(C)=5$, we get $|A^-_2(C)|=12$, so $s=1$ and therefore the only possible type is $(1,4)$. The cases $g(C)=2,3,4$ are analogous.
\end{proof}

Theorem \ref{poss-genera} gives a necessary condition for the genus of $C$. We would like to show that they are sufficient.
For $g(C)=2$ the curve $C$ defines a principal polarisation in its Jacobian. For genus 3, if $f: \tC \ra C$ is an \'etale double covering of a genus 2 curve,
then $g(\tC)=3$ and one has an embedding of $\tC$ in the abelian surface $\widehat{f^*(JC)}$ giving a $(1,2)$ polarisation (see Proposition \ref{prop:cyccov}).
For $g(C)=4$ there is a construction of so called $(1,3)$-theta divisors (see \cite{BS}). It is almost the same as the construction of the curve $C_A$ in the next
section, but it is important to note that constructing $(1,3)$-theta divisors has been technically easier, mostly due to the fact that 2 and 3 are coprime. 

In the following sections we will show two independent constructions that will prove the existence of genus 5 hyperelliptic curves on $(1,4)$ polarised 
abelian surfaces.

\section{Hyperelliptic curves on an abelian surface}

Let $A$ be a $(1,4)$ polarised abelian surface. Fix an isomorphism of $A$ with the abelian surface given by the period matrix $Z\in \ch_2$, 
where $\ch_2$ denotes the Siegel upper half-space. With respect to a standard decomposition $\CC^2 = Z\RR^2 + \RR^2$ there
exists a unique polarising line bundle $L_0$ of  characteristic~$0$ on~$A$. By \cite[Remark  8.5.3]{BL} the space of its global
sections can be identified with a space of classical theta functions with basis
$$
\Th{0}{0}, \  \Th{\omega}{0},\  \Th{2\omega}{0}, \  \Th{3\omega}{0},
$$ 
where $\omega=(0, \frac{1}{4})$. As the
line bundle is symmetric, the $(-1)$ action extends to $H^0(L_0)$. Using \cite[Lemma 8.5.2]{BL} and the fact that the characteristic is 0, we translate Inverse 
Formula \cite[Formula 4.6.4]{BL} to classical theta functions, so $(-1)$
acts on the basis by:
$$
(-1)^*\Th{0}{0}=\Th{0}{0},\quad (-1)^*\Th{\omega}{0}=
\Th{3\omega}{0},\quad
(-1)^*\Th{2\omega}{0}=\Th{2\omega}{0},\quad (-1)^*\Th{3\omega}{0}=
\Th{\omega}{0}.
$$
According to \cite[Corollary 4.6.6]{BL} the dimension of the space of theta functions invariant under the action of $(-1)$ is 3 and the dimension 
of the anti-invariant space is 1.  Thus,
up to a constant, there exists a unique odd theta function ($(-1)$-anti-invariant) in $H^0(L_0)$, denoted by $\theta_A$, namely
$$
\theta_A=\Th{3\omega}{0}-\Th{\omega}{0}.
$$
\begin{defn}
Let  ${C}_A=(\theta_A=0)$ be the zero locus of the theta function  $\theta_A$.% (seen as a section of a line bundle on A).  
\end{defn}
The fact that, up to translation, the curve $C_A$ does not depend on the choice of a period matrix and on characteristic of the line bundle will follow from Theorem \ref{C_A}.
 
The next lemma states some basic properties of the curve $C_A$.
\begin{lem}\label{lem:baspr} 
Let $A$ be a $(1, 4)$ polarised surface and $C_A$ be a curve constructed as above. Then:
\begin{itemize}
\item[(a)] $C_A$ is of arithmetic genus $5$.
\item[(b)] $C_A$ passes through at least twelve $2$-torsion points on $A$.
\item[(c)] If $A$ is a general abelian surface then the curve $C_A$ is smooth. 
\item[(d)] If $C_A$ is smooth then it is a double cover of $\PP^1$
  branched along $12$ points, i.e.\ it is a hyperelliptic curve.
\end{itemize}
\end{lem}
\begin{proof}
Applying the adjunction formula  and Riemann-Roch one obtains
$$
2g_a(C_A) - 2 = C^2 = L^2= 2h^0(L) =8
$$
and hence
$$
g_a(C_A) =\frac{1}{2}(8 + 2) = 5,
$$
which proves (a). Part (b) is a consequence of the proof of \cite[Proposition 4.7.5]{BL}. Using the notation from the proof, one immediately see that 
$C_A$ is an odd divisor, hence $q_{L_0}^{-1}(1)=A_2^-(C_A)$. Moreover,  $L$ is of characteristic 0 and $q_{L_0}$ is of rank 1, so 
$|q_{L_0}^{-1}(1)|=3\cdot4=12$. 

Another way to prove part (b) is to compute explicitly $A_2^-(C_A)$ in the product case (see Remark \ref{rem:12points} for details) and use the fact that
 classical theta functions are continuous on $\CC^2\times\ch_2$.

Since $C_A$ is defined as the zero locus of an explicit theta function,
 (c) is a direct application of Andreotti-Mayer theory: see
\cite[Prop 6]{AM} or \cite[Ch. 6.4]{ACGH} for details.
  
For (d), let $S=A/(-1)$ be the Kummer surface for~$A$. The projection
$A\to S$ is $2:1$, so the image of a smooth symmetric curve is a smooth
curve. From the Hurwitz formula we can find the genus of the image and
the number of branch points:
$$
2\cdot 5 - 2 = 2(2g - 2) + b.
$$
As we know that $b \geq 12$, the only possibility is $g = 0,\ b = 12$.
\end{proof}

\subsection{An explicit computation}\label{sec:expcom}
%\subsection{Product of elliptic curves}
We will describe explicitly the curve $C_A$ as the zero locus of an odd theta function in the case where $A=E\times F$ is a product of elliptic curves 
with a product polarisation. Let
$$
E = \CC/\tau_1\ZZ+\ZZ,\ F = \CC/\tau_2\ZZ+\ZZ,\ 
\Lambda=\left[\begin{array}{cccc} \tau_1& 0& 1& 0\\ 0&\tau_2& 0&
    4\end{array}\right]\ZZ^4,\ A = \CC^2/\Lambda.
$$ 
%Then $A = E\times F$,
%with the product polarisation. 
We can take a standard decomposition
$$
\CC^2 =\left[\begin{array}{cc} \tau_1&
    0\\ 0&\tau_2\end{array}\right]\RR^2 + \RR^2
$$ 
and write theta functions explicitly: for $v=(v_1,v_2)\in\CC^2$ and $\omega=(0,\frac{1}{4})$ we have
\[
\Th{\omega}{0}(v)
= \sum_{l_1,l_2\in\ZZ} \exp \left( \pi il_1^2\tau_1
+ \pi i\left(l_2 +\frac{1}{4}\right)^2\tau_2 
+2\pi i l_1v_1
+ 2\pi iv_2 \left(l_2 +\frac{1}{4}\right) \right).
\]
\[
\Th{3\omega}{0}(v)
= \sum_{l_1,l_2\in\ZZ} \exp\left(\pi il_1^2\tau_1
+ \pi i\left(l_2 +\frac{3}{4}\right)^2\tau_2 
+2\pi i l_1v_1
+ 2\pi iv_2\left(l_2 +\frac{3}{4}\right) \right).
\]

%Recall that classical theta functions are periodic with respect to integers and therefore $\Th{3\omega}{0}(v,Z)=\Th{-\omega}{0}(v,Z)$.
For the computations, let us denote
\begin{align*}
a_l &= \exp\left(\pi il^2\tau_1 + 2\pi i v_1l\right),\\
b_l^\pm &= \exp\left(\pi i \left( l+\frac{1}{2}\pm\frac{1}{4}\right)^2\tau_2 + 2\pi i v_2\left(l+\frac{1}{2}\pm\frac{1}{4}\right)\right).%\\
\end{align*}
Then, since the series converge absolutely
$$
\theta_A = \Th{3\omega}{0}-\Th{\omega}{0}
=\sum_{l_1,l_2}a_{l_1}b^+_{l_2}-\sum_{l_1,l_2}a_{l_1}b^-_{l_2}
=\sum_{l_1}a_{l_1}(\sum_{l_2} b^+_{l_2}-\sum_{l_2} b^-_{l_2}).
$$
For $v_1 = \frac{1}{2} + \frac{1}{2}\tau_1$ we have
\begin{align*}
a_l &= \exp\left(\pi il^2\tau_1+\pi i l\tau_1 + \pi il \right)\\
    &= \exp\left(\pi i\left(l+\frac{1}{2}\right)^2\tau_1-\frac{1}{4}\pi i\tau_1 + \pi il\right)\\
    &= (-1)^l\exp\left(\pi i\left(l+\frac{1}{2}\right)^2\tau_1-\frac{1}{4}\pi i\tau_1 \right).
\end{align*}
Now,
$a_l = -a_{-l-1}$, so $\sum_l a_l = 0$ and therefore for any $v_2$ we
find $\theta_A(\frac{1}{2} + \frac{1}{2}\tau_1,v_2)=0$.  The image of
this component of $\{v\mid \theta_A(v)=0\}$ in $A$ is a curve isomorphic to $F$.

Similar computations can be carried out for $v_2=0,\ v_2=2,\ v_2=\frac{1}{2}\tau_2$
and $v_2=2+\frac{1}{2}\tau_2$. For $v_2=0$ and $v_2=2$, we
have $b^+_l=b^-_{-l-1}$%, since $(l+\frac{1}{2}+\frac{1}{4})^2=(-l-\frac{1}{2}-\frac{1}{4})^2$
.  For $v_2=\frac{1}{2}\tau_2$ and $v_2=2+\frac{1}{2}\tau_2$ we have
$b^+_l=b^-_{-l}$. In all cases we get $\sum_l b_l^+=\sum_l b_l^-$.
The images in $A$ of those zeros are isomorphic to $E$, so by
Lemma~\ref{lem:baspr}(a), we know we have found all the zeros of $\theta_A$.

\begin{rem}\label{rem:12points}
Explicit computations show that in the product case, $C_{E\times F}$ contains twelve 2-torsion points with multiplicity one in 
$E\times F$. Let $F[2]=\{p,q,r,s\}$. Then $F'=F/F[2]$ is an elliptic curve and %we have a natural polarised isogeny $E\times F\lra E\times F'$. Moreover 
$C_{E\times F}$ is an \'etale Klein cover of $E\cup F'$ with $E\cap F'$ being a two torsion point both in $E$ and $F'$ (see Figure \ref{fig:Product-elliptic}). 
\end{rem}

\begin{figure}[h]
 \begin{center}
  \includegraphics[width=11cm,height=3.5cm]{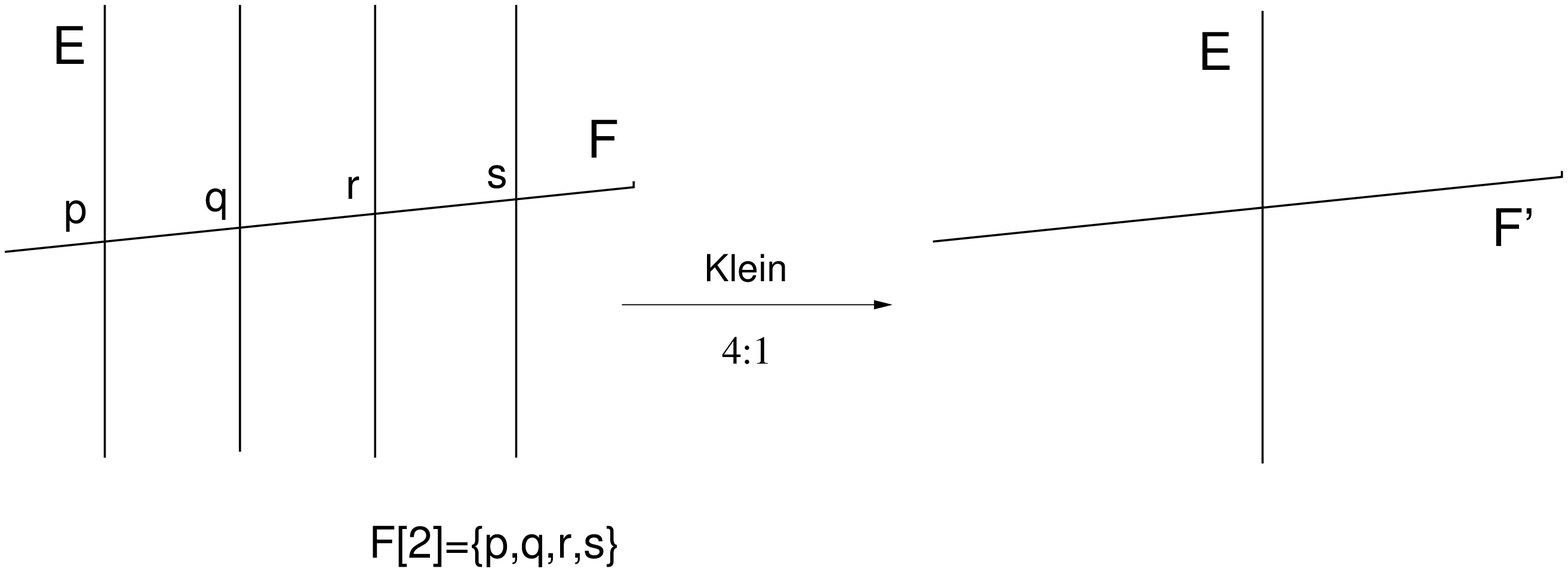}
  \caption{}
  \label{fig:Product-elliptic}
 \end{center}
\end{figure}

The picture in the product case is not a coincidence. Let $Z$ be a period matrix of $A$ and $w_1=(0,2),\ w_2=(\frac{z_{12}}{2},\frac{z_{22}}{2})$ 
be preimages in $\CC^2$ of generators of the subgroup of 2-torsion points in $K(L_0)$. 
%and let $$G=\{0,w_1,w_2,w_1+w_2\}=\ker\phi_{L_0} \cap \frac{1}{2}\Lambda.$$ 
%By direct computation,
One can prove the following technical lemma.
\begin{lem}\label{lem:teccom}
 Let $Z=[z_{ij}]$ be a period matrix of $A$ and $w_1=(0,2),\ w_2=(\frac{z_{12}}{2},\frac{z_{22}}{2})$. Then $$\theta_A(v+w_1)=-\theta_A(v)\ 
 \text{ and }\ \theta_A(v+w_2)=M(Z)\theta_A(v)$$ for some nonzero constant $M(Z)$. %In particular, the curve $C_A$ is invariant with respect to translation by the Klein group of $2$-torsions contained in $K(L_0)$
\end{lem}
\begin{proof}
Both equalities follows from direct technical computations.
One can deduce the first equality from \cite[Remark 8.5.3a]{BL}, as $w_1\in\ZZ^2$ and $\exp(2\pi i\cdot\tfrac{1}{4}\cdot2)=-1$ and the other constants are 0.
The main ingredient to show the second equality  is that $\Th{\omega}{0}(v+w_2)=M(Z)\Th{3\omega}{0}(v)$  and 
$\Th{3\omega}{0}(v+w_2)=M(Z)\Th{\omega} {0}(v)$ for some constant $M(Z)$ that depends only on $Z$.

%We have included Lemma \ref{lem:teccom} for the sake of completeness. 
Note that one can deduce this result from Theorem \ref{thm:hyp}
that is proved indepedently.  
\end{proof}

\begin{cor}\label{cor:invCA}
In particular, the curve $C_A=\{\theta_A=0\}$ is invariant with respect to a subgroup of translations isomorphic to the Klein group, namely 
$K(L_0) \cap A[2]$.
\end{cor}
%\textcolor{red}{It woud be nice to have a bit more explicit the 4 translations  for which $C_A$ is invariant }
%\begin{rem}
%\end{rem}
Now, we would like to prove the uniqueness (up to translation) of such curve. %To do that, we need to 
 \begin{lem}\label{lem:4cop}
 Let $A=E\times F$ be a $(1,4)$ product polarised abelian surface. There are exactly four copies of symmetric curves that passes through twelve $2$-torsion points and all of them are translations by $2$-torsions of $C_{E\times F}$.
 \end{lem}
 \begin{proof}
Denote by $\pi_E$ and $\pi_F$ the corresponding projections. We claim that, any curve in the linear system of a polarising line bundle on $A$ is a union of one fibre of $\pi_E$ and four fibres of $\pi_F$ (counted with multiplicity if needed). To see this, note that by \cite[Lemma 10.1.1]{BL}, the linear system has a fixed component that is a fibre of $\pi_E$. Then, the moving part is given by $\pi_F^*(\cO(4))$ and hence is a union of 4 copies of fibres of $\pi_F$. Actually, we already have seen this fact when we proved that $\theta_{E\times F}$ is a product of two series in two independent variables.

Now, let $C=E_0\cup F_1\cup\ldots\cup F_4$ be the decomposition into the union of fibres. Since $C$ is symmetric, the fixed component $E_0$ has to be the preimage of a $2$-torsion point on $E$. Then, the only case when $C$ passes through twelve 2-torsion points with odd multiplicity is when 
$\{\pi_F(F_i):i=1,\ldots 4\}=F[2]$ is the set of all 2-torsions on $F$. Therefore, there are exactly four copies
of such curves $C$ given by $2$-torsions on $E$ and they are defined by $C_{E\times F}=0$ and its translations by 2-torsion points on $E\times \{0\}$. 
\end{proof}
%Using continuity of Riemmann theta functions on the Siegel space and degenaration argument, we get an immediate corollary.  
%\begin{cor}\label{cor:4cop}
%Let $A$ be a general $(1,4)$ polarised abelian surface. Then there exist exactly four copies of curves that pass through twelve 2-torsion points on $A$ 
%with multiplicity 1. They are all translations of the curve $C_A$ by 2-torsion points.
%\end{cor}

\begin{thm} \label{C_A}
Let $f:C\ra A$ be a smooth hyperelliptic genus 5 curve embedded in a general $(1,4)$ polarised abelian surface $A$. 
Then $C$ is a translation of the curve $C_A$.
In particular, if $L_0=\cO(f(C))$ is of characteristic 0, then $f(C)$ is a translation of $C_A$
by an element of $K(L_0)$. 
\end{thm}
\begin{proof}
By translating if needed, we can assume that $f(p)=0$ for some Weierstrass point $p \in C$. Then, by Lemma 
\ref{lem:WcapA2}, $C$ passes through 
twelve 2-torsion points on $A$ and is defined as the zero locus of some (classical) theta function. Since classical theta functions are holomorphic 
on the Siegel space,  we can degenerate  $A$ to a product of elliptic curves and by Lemma \ref{lem:4cop}, $C$ has to be a translation of $C_A$. The second part
 follows from the definition of $K(L_0)$.
\end{proof}

\begin{rem}
In  \cite{BOPY},  it is proved that the number of hyperelliptic curves in the fixed linear system $h^{A,FLS}_{g,\beta}$ for a $(1,4)$ polarised surface 
$A$ and a genus 5 curve equals 4 (see \cite[Table~1]{BOPY}). It comes from the curve $C_A$ and its translations by elements of $K(L_0)$. 
Although, there are 16 elements of  $K(L_0)$, the curve is invariant with respect to $K(L_0)\cap A[2]$, so there are exactly 4 copies of hyperelliptic 
curves in  the linear system of a $(1,4)$ polarising line bundle on a general abelian surface.  
 \end{rem}
\begin{rem}
In \cite{BLvS} and rewritten in \cite[Section 10.5]{BL}, it is shown that the map attached to the linear series of the polarising line bundle $L_0$ on a general $(1,4)$ polarised abelian surface is birational onto its image. Using the notation from the paper, we have  $\theta_A=X_3-X_1=Z_3$ hence the image of $C_A$ is given by the hyperplane section $Z_3=0$. 
%Moreover, the map $p$ gives the Klein covering of $C_A$ that will be explained better in the next section.
\end{rem}

\section{Klein coverings}

The hyperelliptic curves constructed in \S 3 admit an action of the Klein group $V_4 \simeq \ZZ_2 \times \ZZ_2$.  
In this section we show a necessary and sufficient condition for a Klein covering of a hyperelliptic curve $H$ of genus $g\geq 2$ to be hyperelliptic. 
First, we recall some facts about the Weil pairing on the group $JH[2]$ of 2-torsion points of a hyperelliptic curve that can be found in \cite[Section 5.2]{DO}. 

Let $W=\{w_1, \ldots ,w_{2g+2}\} \subset H$ be the set of Weierstrass points of the curve H. Given a subset $S \subset I=\{1, \ldots, 2g+2\}$ the divisor
\begin{equation} \label{W.pts}
\alpha_S= \sum_{i \in S} w_i - |S| \cdot w_{2g+2}  
\end{equation}
defines an element in $JH[2]$ (where $\alpha_{\{2g+2\}}$ corresponds to zero divisor). Observe that $\alpha_S = \alpha_{I \setminus S}$.  
Denote by $E_g$  the $\FF_2$-vector space of functions $I \ra \FF_2$ having an even number of $0$'s and $1$'s
 modulo the constant functions $\{0,1\}$.  Hence, the elements of $E_g$ are represented by the subsets of even cardinality (up 
to complementary subset) and clearly $E_g \simeq \FF_2^{2g}$.  The correspondence $S \mapsto \alpha_S$ gives an isomorphism 
$E_g \simeq JC[2]$, so all the 2-torsion points on a hyperelliptic Jacobian are of the form \eqref{W.pts}. Moreover, $E_g$ carries a symmetric bilinear form 
$$
e: E_g \times E_g \ra \FF_2, \qquad e(S,T) = |S\cap T| \mbox{ mod } 2,
$$
which is a non-degenerated symplectic form. Under the above isomorphism, this form corresponds to the Weil pairing on the 2-torsion points on $JH$, for details see \cite[Section 5.2]{DO}.
A subspace $G \subset E_g$ is {\it isotropic} with respect to $e$ if $e(\alpha_S,\alpha_T) =  0$ (that is, if the divisors $\alpha_S $ and $\alpha_T$  share
an even number of points) for all $\alpha_S,  \alpha_T \in G $; otherwise the subspace $G$ is called {\it non-isotropic}.  

Now we will make precise what type of coverings arose in Section 3.

\begin{defn}
A \textit {Klein covering} is a couple $(H, G)$ with $H$ a curve of genus $g \geq 1$ and $G$ a subgroup of $JH[2]$ isomorphic to 
the Klein group $V_4=\ZZ_2 \times \ZZ_2$.  If we additionaly assume that $G$ is non-isotropic, we call it a  \textit {non-isotropic Klein covering}.

\end{defn}
The definition is independent of the choice of a basis for  $G$ but in practice we will work with a pair of generators $\eta_1 , \eta_2 $ of $G$.   
In a first step we consider  an  \'etale double covering $h:C\ra H$ given by a 2-torsion point $\eta$.  The following proposition has been proved in \cite{F}
but for sake of clarity we give a proof here.

\begin{prop} \label{farkas-lemma}
Let $H$ be a hyperelliptic curve of genus $g$ and $h:C\ra H$ an \'etale double covering defined by $\eta\in JH[2]$.  Then $C$ is hyperelliptic if and only if 
$\eta= \cO_H(w_1-w_2)$, where $w_1, w_2 \in H $ are Weierstrass points.
\end{prop} 

\begin{proof}
Let $\iota_H$ be the hyperelliptic involution on $H$ and $s$ the automorphism of $C$ exchanging the sheets of the covering.
Suppose that $C$ is hyperelliptic with hyperelliptic involution $ \iota$. Note that the genus of $C$ is $2g-1$ and there are $4g$ Weierstrass points on $C$.
According to the proof of Proposition \ref{prop:hypcov}, $ \iota$ descends to $\iota_H$, so $h$ maps the fixed points of $\iota$ into fixed 
points of $\iota_H$. If a point $p \in C$ is fixed by $\iota$, then $s(p) $ is also fixed by $\iota$, since $s$ and $\iota$ commute.  This implies that there are
exactly $2g$ fibres over Weierstrass points on $H$ containing  all the Weierstrass points of $C$ and 4 points on $C$,
the fibres of the two remaining Weierstrass points, say $w_1, w_2$,  which are fixed points of the involution $\iota s$.  
Since $h^{-1}(w_1) \simeq h^{-1} (w_2)$ are linear equivalent to the hyperelliptic divisor, the 2-torsion point $\eta:=\cO_H(w_1-w_2) \in \Ker h^*$.
This means  that  the double covering $h$ is defined by $\eta$.

For the converse, suppose $h:C\ra H$ is given by a 2-torsion point of the form $\eta:=\cO_H(w_1-w_2)$. Consider the divisor $D:=h^{-1}(w_2)$ of degree 2 on $C$.
By the projection formula we have that 
$$
h^0(C, \cO_C(D) ) = h^0(H, h_*(h^* \cO_C(w_2) \otimes \cO_C) ) = h^0(H, \cO(w_2)) \oplus h^0(H, \cO(w_2) \otimes \eta) = 2 
$$
since $h_*\cO_C = \cO_H \oplus \eta$.  Hence $D$ is a hyperelliptic divisor on $C$ which finishes the proof.
\end{proof}

\begin{figure}[h] 
 \begin{center}
  \includegraphics[width=6.5cm,height=5.5cm]{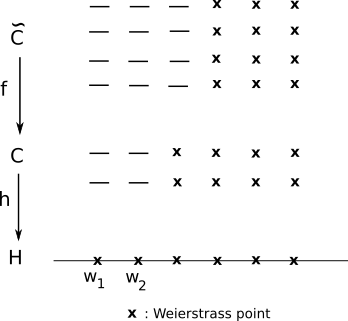}
 \caption{}
  \label{fig:W.pts}
 \end{center}
\end{figure}

\begin{cor}\label{cor:farkas-lemma}
In particular, we have shown the distribution of Weierstrass points on $C$. If $\{w_1,\ldots,w_{2g+2}\}$ is the set of Weierstrass points on H and the covering is defined by $\cO_H(w_1-w_2)$ then the set of Weierstrass points of $C$ is $\{v_{i1},v_{i2}:\ i=3,\ldots,2g+2\}$ and $h(v_{i1})=h(v_{i2})=w_i$, see Figure \ref{fig:W.pts}.
\end{cor}
\begin{rem}
Proposition \ref{farkas-lemma} shows that a hyperelliptic curve admits exactly ${ 2g+2 \choose 2}$ non-trivial  \'etale double coverings which are 
hyperelliptic. In particular, for $g=2$ all the \'etale double coverings are hyperelliptic,   (see \cite[Section 7]{mu}).
\end{rem}

Now, we consider a Klein covering $(H, G)$ with $\eta_1$ and $\eta_2$ generators of $G$ and $g(H) \geq 2$.  Let $h: C \ra H$ be the double 
covering defined by $\eta_1$ and $f:\tC \ra C$ the covering defined by $h^*\eta_2 \in JC[2] \setminus \{ 0\}$. 
Let $s$ be the involution exchanging the sheets of $h$  and $\tau$ the involutions exchanging the sheets of $f$. Then, $s$ lifts to involutions on $\tC$ denoted by $\sigma$ and $\sigma\tau$ and it is easy to check that $\tC$ admits an action of 
$\{\id, \sigma, \tau, \sigma\tau \} \simeq V_4$. Moreover,  the $\tC \ra H$ is Galois with Galois group the Klein group $V_4$.  
\begin{rem}
Note that the curve $\tC$ does not depend on the choice of the generators of $G$, although, the curve $C$ does. 
Choosing a pair of generators gives a path on the commutative diagram:
\begin{equation} \label{diag1}
\xymatrix@R=1.1cm@C=.9cm{
& \tC \ar[d]_{f} \ar[dr]^{2:1} \ar[dl]_{2:1} &\\
C_{\sigma} \ar[dr]_{\eta_2}  & C_{\tau} \ar[d]_{h}^{\eta_1}   & C_{\sigma\tau}  \ar[dl]^{\eta_1 +\eta_2}  \\
& H &
}
\end{equation}
where $C_{\alpha} = \tC / \langle \alpha \rangle$ with $\alpha \in \Aut (\tC)$. 
\end{rem}
\begin{rem}
The necessary and sufficient condition for $s$ to be liftable is that 
the 2-torsion defining $f$ is a pullback of a 2-torsion from $JH$. 
Therefore, we obtain an equivalent definition of the Klein covering:
A Klein covering is a map $f:\tC\ra H$, that is an \'etale Galois covering with the monodromy representation isomorphic to $V_4$.
\end{rem}

Now, we proceed to the main result of this section. Although, we are interested in the case $g(H)=2$, the statement is true for any $g(H)\geq 2$.
 \begin{thm}\label{thm:klehyp}
Let $H$ be a smooth hyperelliptic curve of genus $g\geq 2$ with a subgroup $G=\{0,\eta_1,\eta_2,\eta_1+\eta_2\} \subset JH[2]$ and let $\tC \ra H$ be the Klein covering of (H,G).
 \begin{itemize}
 \item[(a)] If $\tC$ is hyperelliptic then $G$ is non-isotropic with respect to the Weil pairing on $JH[2]$.
 \item[(b)] If  $G$ %= \langle \eta_1, \eta_2 \rangle$ 
 is non-isotropic with $\eta_1$ and $\eta_2$ (and $\eta_1+\eta_2$) being the difference of two Weierstrass points,
 then $\tC$ is hyperelliptic.
 \end{itemize}
 \end{thm}
 
 \begin{proof}
(a) Suppose that $\tC$ is hyperelliptic. By Proposition \ref{prop:hypcov}, $C$ is hyperelliptic hence according to Proposition \ref{farkas-lemma} and 
Corollary \ref{cor:farkas-lemma},  
the 2-torsion point $\eta'_2 := h^*(\eta_2) \in JC[2]$ is of the form $\cO_C(v_{i1} -v_{i2})$ where $v_{i1}$ and $v_{i2}$ are Weierstrass points of $C$ and 
they lie on the fibre $h^{-1} (w_i)$ of a  Weierstrass point $w_i \in H$ for some $i\in\{1,\ldots,2g+2\}$. Since $C$ is hyperelliptic, we also have that  
$\eta_1= \cO_H(w_1 -w_2)$,
for some Weierstrass points, say $w_1, w_2 \in H$. By the distribution of the Weierstrass points of $C$ shown in Corollary \ref{cor:farkas-lemma},  
$w_i \neq w_1$
and $w_i \neq w_2$.  On the other hand,  one can set  $\eta_2= \cO_H(w_1- w_i)$ since
$$
 h^*(\eta_2) = \cO_C( h^{-1}(w_1) -v_{i1} -v_{i2} ) =  \cO_C( v_{i1} -v_{i2} )=\eta_2' 
$$
with $h^{-1}(w_1)$ being the hyperelliptic divisor linear equivalent to $2v_{i1}$.  Hence, the Weil pairing $e(\eta_1, \eta_2) = 1$ since the corresponding
 divisors share the point $w_1$ (see Figure \ref{fig:W.pts} for the case $g=2$).
%One checks easily that also $e(\eta_1+ \eta_2, \eta_i) = 1$, for $i=1,2$, so $V_4$ is non-isotropic.

(b) Suppose that $G$ is non-isotropic and $\eta_1= \cO_H(w_1 -w_2)$, so necessarily $\eta_2$ share a point with $\eta_1$ say 
 $\eta_2= \cO_H(w_1 -w_3)$, with $w_1, w_2, w_3$ Weierstrass points.  By  Proposition \ref{farkas-lemma}, the covering $C \ra H$ defined by 
 $\eta_1$ is hyperelliptic and one checks that $\eta_2' = h^*(\eta_2)$ is the difference of two Weierstrass points, namely $\cO_C( v_{31} -v_{32} )$, 
 therefore the  double covering  $\tC \ra C$ given by $\eta'_2$ is hyperelliptic as well.

\end{proof}
As an immediate consequence we have:
\begin{cor} \label{cor:klehyp}
If $H$ is a smooth curve of genus 2, the Klein covering given by $(H, V_4)$ is hyperelliptic if and only if $V_4$ is non-isotropic with respect to the 
Weil pairing.
\end{cor}
\begin{rem}
Klein coverings of hyperelliptic curves are Galois coverings of maximal degree that can be hyperelliptic. To see this, note that by Proposition 
\ref{prop:hypcov}, the degree of such covering has to be a 2-group. On the other hand by Theorem   \ref{thm:klehyp}, the necessary condition for the 
Klein covering to be hyperelliptic is that the Klein group defining it is non-isotropic. So, if the Galois covering is of degree $> 4$ then for
$\eta_1,\eta_2,\eta_3$ generators of a $\ZZ^3_2$ subgroup of the group defining a covering, we have $e(\eta_1, \eta_2) = e(\eta_1, \eta_3) = 1$, hence
$e(\eta_1, \eta_2-\eta_3) = 0$.  Therefore, for any covering 
of degree strictly bigger than $4$ there exists an isotropic Klein subgroup. Thus at least one of the quotient curves is non-hyperelliptic, which 
contradicts Proposition \ref{prop:hypcov}. This fact has been already proved in \cite{Mac} using Fuchsian groups.
\end{rem}

%\begin{rem}
%In the proof of Theorem \ref{thm:klehyp} the covering $\tC$ is independent of the choice of the generators of $V_4$:  if $\eta_1$ and $\eta_2$ 
%are difference of Weierstrass point, $\eta_1 + \eta_2$ is also of this form. 
%As before, let  $\sigma$ and $\tau$ be the corresponding fixed point free involutions on $\tC$. Choosing a pair of generators gives a path on the commutative diagram:
%\begin{equation} \label{diag1}
%\xymatrix@R=1.1cm@C=.9cm{
%& \tC \ar[d]_{f} \ar[dr]^{2:1} \ar[dl]_{2:1} &\\
%C_{\sigma} \ar[dr]_{\eta_2}  & C_{\tau} \ar[d]_{h}^{\eta_1}   & C_{\sigma\tau}  \ar[dl]^{\eta_1 +\eta_2}  \\
%& H &
%}
%\end{equation}
%where $C_{\alpha} = \tC / \langle \alpha \rangle$ with $\alpha \in \Aut (\tC)$.
%\end{rem}

\subsection{The second construction of $C_A$} \label{2nd-const}
From the perspective of this section, we would like to realise $C_A$ as a Klein covering.
Let $(A, L)$ be a general $(1,4)$ polarised abelian surface. Let $\eta, \gamma$ be any generators of $K(L) \simeq \ZZ_4\times\ZZ_4$. Denote by 
$K=\{0,2\gamma, 2\eta, 2\gamma+2\eta\}=K(L)\cap A[2]$. Let 
$$
A_{2\eta}=A/\langle 0,2\eta\rangle, A_{\eta\gamma}=A/\langle 0, 2\eta+2 \gamma\rangle  \text{ and } 
A_{2\gamma}=A/\langle 0,2\gamma\rangle.
$$

Since $|K|=4$, the surface $A/K$ is principally polarised and because $A$ is general, we can assume $A/K=JH$ is the Jacobian of a smooth 
genus 2 curve $H$. We have the following diagram of polarised isogenies:
\begin{equation} \label{diagsurfaces}
\xymatrix@R=1.1cm@C=.9cm{
& A \ar[d]_{2:1} \ar[dr]^{2:1} \ar[dl]_{2:1} &\\
A_{2\eta} \ar[dr]_{2:1}  & A_{\eta\gamma} \ar[d]_{2:1}   & A_{2\gamma}  \ar[dl]^{2:1}  \\
& JH &
}
\end{equation}
Now, we will build up curves as preimages of respective projections. Let $\pi:A\lra JH$. Then $\pi(\eta), \pi(\gamma), \pi(\eta+\gamma)\in JH[2]$ are  
2-torsion points and they define \'etale double coverings of $H$ denoted $C_{\eta}, C_{\eta+\gamma}, C_{\gamma}$. Moreover, the group $G=\{0, \pi(\eta), 
\pi(\gamma), \pi(\eta+\gamma)\}$ defines an \'etale Klein covering of $H$. The coverings can be realised as restrictions of the respective projections
\begin{equation} \label{diagcurvesandsurfaces}
\xymatrix@R=1.1cm@C=.9cm{
& \pi^{-1}(H)\subseteq A \ar[d]_{2:1} \ar[dr]^{2:1} \ar[dl]_{2:1} &\\
C_{\eta}\subseteq A_{2\eta} \ar[dr]_{2:1}  &C_{\eta+\gamma} \subseteq A_{\eta\gamma} \ar[d]_{2:1}   &C_{\gamma}\subseteq A_{2\gamma}  \ar[dl]^{2:1}  \\
& H\subseteq JH &
}
\end{equation}
\begin{rem}
Note that the groups $K$ and $\pi(K)$ do not depend on the chosen generators so, up to permutation, the curves do not depend on them either. Moreover, by construction, $JH/G$ is the dual surface to $A$, so it has to be $(1,4)$ polarised.
\end{rem}

Let us describe the notion of being isotropic and non-isotropic Klein subgroups in terms of types of polarisation on quotient abelian surfaces in the following lemma.
\begin{lem}\label{lem:polarisations}
Let $(JH,\Theta)$ be a general principally polarised surface and let $G\simeq V_4$ be a subgroup of $JH[2]$. Then
\begin{itemize}
\item[(a)] $JH/G$ is principally polarised if and only if $G$ is isotropic.
\item[(b)] $JH/G$ is of type $(1,4)$ if and only if $G$ is non-isotropic.
\end{itemize}
\end{lem}
\begin{proof}
%Denote by $\pi_G:JH\lra JH/G$ the canonical projection. 
Recall that the Weil pairing coincide with the notion of the commutator map $e^{2\Theta}$ introduced in \cite[Section 6.3]{BL} for $2\Theta$ being twice 
the principal polarisation on the Jacobian. Then, (a) follows directly from \cite[Cor 6.3.5]{BL}, since the only type of polarisation 
that can be pulled back by a degree 4 isogeny to a $(2,2)$ type is principal.

As for (b), note that $G$ is isotropic with respect to $e^{4\Theta}$, so the only possible type that can be pulled back to a $(4,4)$ type is 
$(1,4)$, since by (a) we have excluded the $(2,2)$ type.  
%and assume that $X/G$ is polarised of some type $(1,d)$ with $d\in \ZZ_+$.
\end{proof}

Now, we can connect Theorem \ref{C_A} and Theorem \ref{thm:klehyp} and show that two constructions of hyperelliptic curves coincide. 
In particular, we reprove the fact that $C_A$ is invariant with respect to $K(L_0)\cap A[2]$.

\begin{thm}\label{thm:hyp}
A smooth hyperelliptic curve of genus 5 can be embedded into $(1,4)$ polarised abelian surface if and only if it is a non-isotropic \'etale Klein covering of a genus 2 curve.% can be realised as the restriction of an isogeny like in  
%Diagram \ref{diagcurvesandsurfaces}. 
In particular, after possible translation on JH, we proved that $\tC=\pi^{-1}(H)=C_A$.
\end{thm}
\begin{proof}
Let $f:\tC\lra H$ be an \'etale Klein covering defined by a non-isotropic subgroup. Then we can build Diagram \ref{diagcurvesandsurfaces} 
from downstairs, hence $\tC$ can be embedded in a 
(1,4) polarised surface, say $A$, and its image is invariant with respect to the Klein subgroup of translations. By possible translation of $H$, we can 
assume that $\tC$ is symmetric and defines the line bundle of characteristic $0$.
Then, by Theorem \ref{C_A}, $\tC=C_A$.

The second implication follows from the fact that every $(1,4)$ polarised surface can be constructed in this way. 

%On the other hand, by Corollary 
%\ref{cor:invCA} the curve $C_A$ is also invariant and therefore one can consider the image of $C_A$ in $JH$. Since $JH$ is a Jacobian, the image of 
%$C_A$ has to be equal to $H$, so $\tC=C_A$ (up to a possible translation). %Then,  by Lemma \ref{lem:baspr}  $C_A$ is hyperelliptic. 
\end{proof}

There is a natural involution on the moduli of abelian surfaces, namely dualisation. The following remark describes the corresponding involution 
on the Klein coverings.

\begin{rem}\label{rem:dual}
Let $G$ be a non-isotropic Klein subgroup of JH[2]. Since the Weil pairing is non-degenerate, there exists a unique orthogonal complement to $G$ and 
it is also non-isotropic Klein group. It is denoted by $\widehat{G}$, since it can be seen as the dual to $G$ when using the canonical identification with the 
dual surface.

Then, by construction, the Klein coverings defined by $(H,G)$ and $(H,\widehat{G})$ are embedded into dual abelian surfaces.
\end{rem}
\section{Decomposition of $J\tC$}

In order to have a better understanding of the decomposition of $J\tC$, we introduce the following notation.
\begin{defn}
Let $X$ be an %principally polarised 
%!!!!! it seems I do not need principal polarisation although it needs to be checked !!!!!
abelian variety and $M_i$ abelian varieties such that there exist embeddings $M_i\hookrightarrow X$ for $i=1,\ldots,k$.
We write $$X=M_1 \boxplus M_2\boxplus\ldots\boxplus M_k$$ if $\epsilon_{M_1}+\epsilon_{M_2}+\ldots\epsilon_{M_k}=1$, where 
$\epsilon_{M_i}$ are the associated symmetric idempotents. In particular,  $X=M\boxplus N$ if and only if $(M,N)$ is a pair of complementary abelian subvarieties of $X$.
%then $X=M\boxplus N$. 

\end{defn}
\begin{prop}\label{prop:absub}Let $X$ be an abelian variety.
\begin{itemize}
\item[(a)] If $X$ is non-simple then there exist $M$ and $N$ such that $X=M\boxplus N$.
\item[(b)] If $X=M\boxplus N$ and $M=M_1\boxplus M_2$ then  $X=M_1\boxplus M_2\boxplus N$.
\item[(c)] If $M_i$ are simple and $\Hom(M_i,M_j)=0$ for $i\neq j$ then the set $\{0,X, M_i:i=1,\ldots,k\}$ is the set of all abelian subvarieties of $X$.
In particular, all $M_i$'s and the presentation are unique (up to permutation). 
\end{itemize}
\end{prop}
\begin{proof}
Part (a) and (b) are trivial. As for the last, let $N$ be a non-zero abelian subvariety embedded in $X$. Consider the restrictions of Norm 
endomorphism $\Nm_{M_i}$ to $N$ and denote them by $\Nm_i$. Since $N$ and $M_i$ are simple, $\Nm_i$ is an isogeny or the zero map. 
From the assumption that $\sum_i\epsilon_{M_i}=1$, we get that the $M_i$'s generate $X$ as a group, so at least one of the $\Nm_i$'s is non-zero. 
On the other hand if we had two isogenies $\Nm_i$ and $\Nm_j$, $i\neq j$ then we would get an isogeny $M_i\ra M_j$ which contradicts 
$\Hom(M_i, M_j)=0$. 
This means that exactly one $\Nm_i$ is non-zero, hence $N\subset M_i$ (as subvarieties in $X$) and from simplicity, $N=M_i$.
\end{proof}
\begin{rem}
Part (a) and (c) of Proposition \ref{prop:absub} can be seen as stronger forms of Poincare Reducibility Theorems.

For an abelian variety $M$ embedded in $X$ if there may be any confusion what is the type of restricted polarisation from $X$ to $M$, we will write
 $M^{(d_1,\ldots, d_{\dim_M})}$ where  $(d_1,\ldots, d_{\dim_M})$ is the type of the restricted polarisation from $X$ to $M$. 
\end{rem}
 
Let $A$ be a general $(1,4)$ polarised surface and $\tC$ be the hyperelliptic curve embedded in $A$. We denote by $\iota$  the hyperelliptic involution
on $\tC$. By results of Section 4, $\tC$ admits an action of the group 
$$
\ZZ_2^3 \simeq  \langle \iota, \sigma,  \tau \ :  \  \iota^2=\sigma^2 =\tau^2=1, \iota\sigma=\sigma \iota,  \tau\sigma =\sigma \tau,  \tau \iota=\iota\tau 
\rangle\subset \Aut{\tC},
$$
where $\sigma$ and $\tau$ are the fixed point free involutions. As before, for any $\alpha\in\Aut(\tC)$, we denote the quotient curve by 
$C_\alpha=\tC/\langle\alpha\rangle$. Since $\sigma, \tau$ and $\sigma\tau$ are fixed point free, we have that $C_{\sigma}, C_{\tau},C_{\sigma\tau}$ 
are the quotient curves of $\tC$ of genus 3 embedded in  $A_{\sigma}, A_{\tau}, A_{\sigma\tau}$. 

Corollary \ref{cor:farkas-lemma} shows that $|\Fix(\iota k)|=4$ for any $k\in\{\sigma, \tau, \sigma\tau\}$, hence  by Hurwitz formula we get that 
$C_{\iota\sigma}, C_{\iota\tau}, C_{\iota\sigma\tau}$  are quotient curves of genus 2. Moreover let 
 \begin{equation}\label{3-elliptic}
 E_{\sigma}=\tC/\langle \sigma, \iota\tau \rangle, \quad E_{\tau}=\tC/\langle \tau, \iota\sigma \rangle, \quad E_{\sigma\tau}=\tC/\langle \sigma\tau,
 \iota\sigma \rangle
\end{equation}
 be elliptic curves given by quotients by 4-elements subgroups. We also have  $H=\tC/\langle \sigma, \tau\rangle$ and rational curves that are the 
 quotient curves for any subgroup that contains $\iota$. In this way, we investigated all possible subgroups of $\ZZ^3_2$.

For any $k,l\in\{\sigma, \tau, \sigma\tau\},\  l\neq k$ we have the following quotient maps:
$$
\tC \ra C_k\ra E_{k}\ \text{  and }\ C_{\iota k}\ra E_l.
$$
Note that all obtained curves are hyperelliptic (or elliptic).

\begin{cor}\label{cor:jacdiv}
Applying Propositions \ref{prop:cyccov}, \ref{prop:doubra} and \ref{prop:douhyp} to the quotient maps we get the following decompositions of 
Jacobians of quotient curves:
$$
JC_{\iota\sigma}=E_{\tau}^{(2)}\boxplus E_{\sigma\tau}^{(2)},\ JC_{\iota\tau}=E_{\sigma}^{(2)}\boxplus E_{\sigma\tau}^{(2)},
\ JC_{\iota\sigma\tau}= E_{\sigma}^{(2)}\boxplus E_{\tau}^{(2)},\ 
$$ 
$$
JC_{\sigma}=\widehat{A}_{\sigma}^{(1,2)}\boxplus E_{\sigma}^{(2)},\ JC_{\tau}=\widehat{A}_{\tau}^{(1,2)}\boxplus E_{\tau}^{(2)},\ 
JC_{\sigma\tau}= \widehat{A}_{\sigma\tau}^{(1,2)}\boxplus E_{\sigma\tau}^{(2)}
$$
\end{cor}

The main result of this section is the following theorem.
\begin{thm}\label{thm:jacc_a}
Let $A$ be a general $(1,4)$ polarised abelian surface and $\tC$ the hyperelliptic curve embedded in $A$. Then
$$J\tC=\widehat{A}^{(1,4)}\boxplus E^{(4)}_{\sigma} \boxplus E^{(4)}_{\tau}\boxplus E^{(4)}_{\sigma\tau}.$$

In this presentation we can see images of Jacobians of all quotient curves and Pryms of coverings. For example:
$$
\widehat{A}^{(1,4)}=\im JH=JH/V_4,
$$
$$
 E^{(4)}_{\sigma} \boxplus E^{(4)}_{\tau}\boxplus E^{(4)}_{\sigma\tau}=P(\tC/H)^{(1,1,4)},
$$
$$
 \widehat{A}^{(1,4)}\boxplus E^{(4)}_{\sigma}=\im JC_{\sigma}=JC_{\sigma}/\ZZ_2=P(\tC/C_{\iota\sigma})^{(1,2,2)},
 $$
 $$ 
  E^{(4)}_{\sigma} \boxplus E^{(4)}_{\tau}=\im JC_{\iota\sigma\tau}=JC_{\iota\sigma\tau}^{(2,2)}=P(\tC/C_{\sigma\tau})^{(2,2)}.
  $$
\end{thm}
\begin{proof}
Since $\tC$ is an \'etale double cover of $C_\sigma$ and both are hyperelliptic, Proposition \ref{prop:douhyp} gives $J\tC=JC_{\sigma}/\ZZ_2
\boxplus JC^{(2,2)}_{j\sigma}$. Applying Proposition \ref{prop:absub} and Corollary \ref{cor:jacdiv} one obtains the result.
\end{proof}
\begin{rem}
Theorem \ref{thm:jacc_a} is a more detailed version of \cite[Theorem 6.3.iv]{RR} in the particular case, when the base curve has genus 2, 
the coverings are \'etale and defined by non-isotropic Klein group.
\end{rem}

\end{document}